\documentclass[12pt]{extarticle}
\usepackage{extsizes}
\usepackage[reqno]{amsmath} 
\usepackage{amsmath}
\usepackage{amssymb}
\usepackage{amsthm}
\usepackage{amscd}
\usepackage{amsfonts}
\usepackage{graphicx}
\usepackage{dsfont}
\usepackage{fancyhdr}
\usepackage{enumerate}
\usepackage{youngtab}
\usepackage[utf8]{inputenc}
\usepackage{comment}

\usepackage{subfigure}
\usepackage[margin=1.5in]{geometry}
\usepackage{graphicx}
\usepackage[noend]{algorithmic}
\usepackage{algorithm}
\usepackage{color}
\usepackage{hyperref}
\usepackage{notation}

\theoremstyle{definition}
\newtheorem{corollary}{Corollary}

\newtheorem{lemma}[corollary]{Lemma}

\newtheorem{theorem}[corollary]{Theorem}

\begin{document}

\AtEndDocument{%
  \par
  \medskip
  \begin{tabular}{@{}l@{}}%
    \textsc{Marcel K. de Carli Silva}\\
    \textsc{Dept. of Computer Science - USP - São Paulo, Brazil}\\
    \textit{E-mail address}: \texttt{mksilva@ime.usp.br} \\ \ \\
    \textsc{Gabriel Coutinho}\\
    \textsc{Dept. of Computer Science - UFMG - Belo Horizonte, Brazil}\\
    \textit{E-mail address}: \texttt{gabriel@dcc.ufmg.br} \\ \ \\
    \textsc{Rafael Grandsire}\\
    \textsc{Dept. of Computer Science - UFMG - Belo Horizonte, Brazil}\\
    \textit{E-mail address}: \texttt{rafael.grandsire@gmail.com}
  \end{tabular}}

\title{An eigenvalue bound for the fractional chromatic number}
\author{Marcel K. de Carli Silva, Gabriel Coutinho\footnote{gabriel@dcc.ufmg.br} , Rafael Grandsire}
\date{\today}
\maketitle
\vspace{-0.8cm}
\begin{abstract}
    We show that Hoffman's sum of eigenvalues bound for the chromatic number is at least as good as the Lovász theta number, but no better than the ceiling of the fractional chromatic number. In order to do so, we display an interesting connection between this sum of eigenvalues bound and a generalization of the Lovász theta number introduced by Manber and Narasimhan in 1988.
\begin{center}
\textbf{Keywords}
\end{center}
Lovász theta function, sum of eigenvalues bound, chromatic number.

\end{abstract}

\section{Introduction}

Let $G$ be a simple graph on $n$ vertices, with adjacency matrix $\aA$, and let $\lambda_1 \geq \lambda_2 \geq \cdots \geq \lambda_n$ denote its eigenvalues. We denote the chromatic number of $G$ by $\chi(G)$. Hoffman \cite{HoffmanEvalsColouring} showed in 1970 that if $n \geq m \geq \chi(G) \geq 2$, then
\[
	\lambda_1 + \lambda_{n} + \lambda_{n-1} + \cdots + \lambda_{n-m+2} \leq 0.
\]
This provides a lower bound for $\chi(G)$, as this sum tries to be positive for small values of $m$. Recently in \cite{WocjanElphickDarbariSpectralBoundsQuantumPart2} it was shown that the same bound holds for the quantum chromatic number $\chi_q(G)$ replacing the role of $\chi(G)$. This bound is also trivially weaker than its more famous cousin $\chi(G) \geq 1 - (\lambda_1 / \lambda_n)$, which has been more extensively exploited in connection to other graph parameters --- see for instance Bilu \cite{Bilu}. 

Our main result in this paper is that the bound also holds for the ceiling of $\chi_f(G)$, the well-known fractional chromatic number, even when the non-zero entries of the adjacency matrix $\aA$ are allowed to be different than $1$. This result along with our other original results are in Section \ref{sec:chi}.

The proof of Hoffman's original result uses the well-known Cauchy's Interlacing Theorem, whose countless applications to spectral graph theory have been found since the seminal work of Haemers \cite{HaemersInterlacing}. The proof requires that entries of $\aA$ corresponding to non-edges be equal to $0$, but makes no requirement on the entries corresponding to edges. Diagonal entries can typically be tweaked inside the proof. So, versions of this theorem exist for other matrices that encode adjacency, for instance, Coutinho, Grandsire and Passos~\cite{CoutinhoGrandsirePassos} exploited consequences of this theorem applied to the normalized Laplacian matrix.

This freedom in choosing entries corresponding to edges leads naturally to an optimization opportunity, and more often than not, at this point, semidefinite optimization comes into play. Lovász \cite{LovaszShanon} introduced in 1979 a graph parameter which is now known as the Lovász theta number of $G$, which can be defined as the optimum value of the semidefinite program (SDP)
\[
	\vartheta(G) = \max \{ \langle \jJ , \xX \rangle : \tr \xX = 1,\ \xX_{ij} = 0 \ \forall ij \in E(G),\ \xX \succcurlyeq \0\},
\]
and showed the remarkable relation $\alpha(G) \leq \vartheta(G) \leq \chi(\ov{G})$, where $\alpha(G)$ stands for the stability number of $G$, that is, the size of the largest coclique in $G$. Here we are denoting the all-ones matrix by $\jJ$, and $\tr$ is the trace. The theory behind this parameter has been further developed and explored, notably with the introduction of the theta function and the theta body by Grötschel, Lovász and Schrijver in  \cite{LovaszGrotschelSchrijverTheta}; and extensively surveyed---famously by Knuth in \cite{KnuthSandwich}.

One interesting generalization of $\vartheta(G)$ was proposed by Narasimhan and Manber \cite{NarasimhanManberJournal,NarasimhanManberTechRep}, and further studied in \cite{shor2002lagrangian,kuryatnikova2020maximum,SotirovSinjorgo}. Let $\alpha_k(G)$ denote the number of vertices of a largest $k$\nobreakdash-colourable induced subgraph of $G$. Let $\lambda_{i}(\mM)$ denote the $i$th largest eigenvalue of a (symmetric) matrix $\mM$. They defined the parameter
\[
	\vartheta_k(G) = \min \left\{ \sum_{i = 1}^k \lambda_{i}(\jJ + \xX) : \xX_{ij} = 0 \ \forall ij \notin E(G),\ \xX^\T = \xX \right\},
\]
and showed that $\alpha_k(G) \leq \vartheta_k(G)$ (we have decided to denote by $\vartheta_k(G)$ what, in their original paper, would have been $\vartheta_k(\ov{G})$ for the sake of keeping consistency with the standard use of $\vartheta(G)$). The connection between Hoffman's result and this parameter was perhaps foreseen by Mohar and Poljak \cite[Corollary~4.16]{MoharPoljak}, who pointed out that if $k \geq \chi(G)$,  then $n \leq \alpha_k(G)$, and thus $n \leq \vartheta_k(G)$. In~Section~\ref{sec:chi}, we will show explicitly how the optimization of Hoffman's result is related to $\vartheta_k$, and we will also show that this is always at least as good as $\lceil \vartheta(\ov{G})\rceil$ but no better than $\lceil\chi_f(G)\rceil$ as a lower bound for $\chi(G)$.

\section{The \texorpdfstring{$\vartheta_k$}{ϑₖ} function} \label{sec:basics}

Let $G = (V,E)$ be a simple graph with vertex set $V$ and edge set $E$, and let $\ww \in \Rds_+^V$, that is, a non-negative vector indexed by $V$. We denote by $\sqrt{\ww}$ the vector obtained by taking the square root at each entry of $\ww$. The set of real symmetric matrices with rows and columns indexed by $V$ is denoted by $\Sds^V$, and if $\mM,\nN \in \Sds^V$, we write $\mM \succcurlyeq \nN$ to mean that $\mM - \nN$ is positive semidefinite. Define, for $0 \leq k \leq |V|$,
\begin{equation}
	\vartheta_k(G;\ww) = \min \left\{ \sum_{i = 1}^k \lambda_{i}(\sqrt{\ww}\sqrt{\ww}^\T + \zZ) : \zZ \in \Sds^V ,\, \zZ_{ij} = 0 \ \forall ij \notin E(G)\right\}. \label{varthetaKoriginal}
\end{equation}

Following a well-known result due to Ky Fan \cite{KyFanEigenvalueSum}, the optimization problem above can be written as the semidefinite program
\sdp
{\vartheta_k(G;\ww) = \hspace{2cm} \min}
{k \eta + \tr \yY\label{varthetaKprimal}}
{
\eta \in \Rds,\ \yY,\zZ \in \Sds^V, \nonumber \\
& \zZ_{ij} = 0 \ \forall ij \notin E(G), \nonumber \\
& \yY - (\sqrt{\ww}\sqrt{\ww}^\T + \zZ) + \eta \iI \succcurlyeq \0, \nonumber \\
& \yY \succcurlyeq \0,\nonumber }
and, from SDP strong duality,
\sdp
{\vartheta_k(G;\ww) = \hspace{2cm} \max}
{\langle \sqrt{\ww}\sqrt{\ww}^\T, \xX \rangle \label{varthetaKdual}}
{
\xX \in \Sds^V, \nonumber \\
& \tr \xX = k , \nonumber \\
& \xX_{ij} = 0 \ \forall ij \in E(G), \nonumber \\
& \0 \preccurlyeq \xX \preccurlyeq \iI, \nonumber}
as both optimization problems have Slater points, except for \eqref{varthetaKdual} when $k= 0$ or $k = |V|$. These formulations have been introduced by Alizadeh in \cite{AlizadehInteriorPointSDP}. Note that~\eqref{varthetaKprimal} and~\eqref{varthetaKdual} define \(\vartheta_k(G;\ww)\) for any \emph{real number} $k$ such that $0 \leq k \leq |V|$.

Fix $\ww \in \Rds_+^V$. If $\zZ$ is an optimal solution for \eqref{varthetaKoriginal}, assume we have the spectral decomposition
\begin{align}\label{eq:specdecompZ}
	\sqrt{\ww}\sqrt{\ww}^\T + \zZ = \sum_{i = 1}^n \lambda_i \vv_i \vv_i^\T.
\end{align}
With $\eta = \lambda_k$, note that the following choice for $\yY$ provides an optimal solution of \eqref{varthetaKprimal}:
\begin{align}\label{eq:specdecompY}
	\yY = \sum_{i = 1}^k (\lambda_i - \lambda_k) \vv_i \vv_i^\T.
\end{align}

If $\xX$ is optimum for \eqref{varthetaKdual}, then complementary slackness implies that $\langle \xX , \yY \rangle = \langle \iI , \yY \rangle$, hence $\langle \iI - \xX , \yY \rangle = 0$. As $\iI - \xX \succcurlyeq \0$, a well-known result due to Schur implies that $\yY = \xX \yY = \yY\xX$, leading to the partial correspondence between eigenvectors of optimal solutions of \eqref{varthetaKprimal} and \eqref{varthetaKdual}.

The recent papers \cite{kuryatnikova2020maximum,SotirovSinjorgo} have explored further properties of the parameter $\vartheta_k$ which should generalize in many cases to the weighted function $\vartheta_k(G;\ww)$. For instance, note that $\vartheta_k(G;\ww) \leq \ww^\T\1$ because $\eta = 0$, $\zZ = \0$ and $\yY = \sqrt{\ww}\sqrt{\ww}^\T$ are feasible for \eqref{varthetaKprimal}. We will also make use of the following lemma; see~\cite[Proposition 2]{SotirovSinjorgo}.

\begin{lemma} \label{lem:monotonicity}
Let $G$ be a graph with $n$ vertices. Let ${0 \leq k \leq \ell \leq n}$, not necessarily integers, and let $\ww \in \Rds_+^{V(G)}$. Then
\begin{equation*}
  \vartheta_k(G;\ww)\leq \vartheta_\ell(G;\ww).
\end{equation*}
\end{lemma}
\begin{proof}
We may assume that $k < \ell$. With $\xX$ feasible for the formulation~\eqref{varthetaKdual} of $\vartheta_k(G;\ww)$, let
\[
    \zZ = \left(1 - \frac{\ell-k}{n-k} \right) \xX + \frac{\ell-k}{n-k} \iI.
\]
It is easy to see that $\zZ$ is feasible for the formulation~\eqref{varthetaKdual} of $\vartheta_{\ell}(G;\ww)$ with objective value
    \begin{align*}
	    \big\langle \sqrt{\ww}\sqrt{\ww}^\T, \zZ \big\rangle
            & = \left(1 - \frac{\ell-k}{n-k}\right) \big\langle \sqrt{\ww}\sqrt{\ww}^\T, \xX \big\rangle + \frac{\ell-k}{n-k} \ww^\T \1 \geq \big\langle \sqrt{\ww}\sqrt{\ww}^\T, \xX \big\rangle
    \end{align*}
since $\ww^\T \1 \geq \vartheta_k(G;\ww) \geq \big\langle \sqrt{\ww}\sqrt{\ww}^\T, \xX \big\rangle$.
\end{proof}

\section{Bound for the fractional chromatic number} \label{sec:chi}

Let now $h(G)$ be the smallest integer $m \geq 1$ such that, for all $\zZ \in \Sds^V$ such that $\zZ_{ij} = 0 \ \forall ij \notin E(G)$, we have
\[\lambda_1(\zZ) + \lambda_n(\zZ) + \dotsb + \lambda_{n - m + 2}(\zZ) \leq 0.
\]
Note that if $\zZ$ is fixed to be the adjacency matrix, this integer $m$ has been called $1+\kappa$ in \cite{WocjanElphickDarbariSpectralBoundsQuantumPart2}. We have decided to introduce the notation $h(G)$ to make it explicit that we are dealing with a different, optimized version. 

As a consequence of the proof of Hoffman's theorem (see for instance \cite[Proposition 3.6.3]{BrouwerHaemers} for a proof), it follows that
\[h(G) \leq \chi(G).\]

It is then natural to ask whether this parameter is at least as good as $\vartheta(\ov{G})$ to lower bound $\chi(G)$. The answer is affirmative.

\begin{theorem}
\label{thm:1}
Let $h(G)$ be as defined above. Then
\[
	\lceil \vartheta(\ov{G}) \rceil \leq h(G) \leq \chi(G).
\]
\end{theorem}
\begin{proof}
We may assume that $G$ has at least one edge. In one of the several equivalent characterizations of $\vartheta(G)$ in \cite{LovaszShanon}, Lovász showed that
\[
\vartheta(\overline{G}) = \max \left\{ 1 - \frac{\lambda_1(\zZ)}{\lambda_n(\zZ)} : \zZ \in \Sds^V \setminus \{0\},\ \zZ_{ij} = 0 \ \forall ij \notin E(G) \right\}.
\]
Let $\zZ^* \in \Sds^V$ be an optimal solution. Then $\lambda_n(\zZ^*)(\vartheta(\ov{G})-1) + \lambda_1(\zZ^*) = 0$. Thus
\[\lambda_n(\zZ^*) (\lceil \vartheta(\ov{G}) \rceil -2) + \lambda_1(\zZ^*) > 0. \]
Hence
\[\lambda_n(\zZ^*) + \lambda_{n-1}(\zZ^*) + \dotsb + \lambda_{n - (\lceil \vartheta(\ov{G}) \rceil - 1) + 2}(\zZ^*) + \lambda_1(\zZ^*) > 0, \]
therefore
\[ h(G) \geq \lceil \vartheta(\ov{G}) \rceil . \qedhere\]
\end{proof}

Note that the proof shows that, if $G = (V,E)$ has at least one edge and  $m = \lceil \vartheta(\ov{G}) \rceil -1$, there is $\zZ \in \Sds^V$ such that $\zZ_{ij} = 0$ for each $ij \notin E$ and
\begin{equation}
	\label{eq:1}
	\lambda_1(\zZ) + \lambda_n(\zZ) + \dotsb + \lambda_{n - m + 2}(\zZ) > 0.
\end{equation}

The parameter $h(G)$ appears to be interesting, but it is not even clear how, or even if, it can be computed. We will show below how $h(G)$ is related to the function $\vartheta_k(G;\ww)$.

\begin{theorem}
\label{thm:maxtheta}
	Let $G$ be a graph, and $h(G)$ be defined as above. In both $\min$ and $\max$ below, $k$ is an integer between $1$ and $n$. Then
        \begin{equation}
          \label{eq:2}
          \min\{ k : \exists~\ww \geq \0, \ww \neq 0, \vartheta_k(G;\ww) \geq \ww^\T \1 \} \ \leq \ h(G)
        \end{equation}
	and
        \begin{equation}
          \label{eq:3}
          h(G) \ \leq \ \min\{ k : \forall~\ww \geq \0, \vartheta_k(G;\ww) \geq \ww^\T \1 \}.
        \end{equation}
\end{theorem}
\begin{proof}
	We show~\eqref{eq:2}. Let $m$ be the $\min$ in the statement minus one. We may assume $m \geq 1$, otherwise $m = 0$ and the inequality follows immediately. Consider the optimization problem
\sdp
{\max}
{\langle \sqrt{\ww}\sqrt{\ww}^\T, \xX \rangle }
{
\xX \in \Sds^V, \nonumber \\
& \tr \xX = m , \nonumber \\
& \xX_{ij} = 0 \ \forall ij \in E(G), \nonumber \\
& \0 \preccurlyeq \xX \preccurlyeq \iI, \nonumber \\
& \ww \geq 0, \nonumber \\
& \ww^\T \1 = 1. \nonumber}
        The feasible region is compact and the objective function is continuous, so an optimum solution $(\xX,\ww)$ exists.  By conjugating $\xX$ by a suitable diagonal matrix of entries $\pm 1$, we can assume there is an eigenvector $\vv$ corresponding to $\lambda_1(\xX)$ with $\vv \geq \0$.  Hence, $\sqrt{\ww}$ is an eigenvector for~$\xX$ with eigenvalue strictly smaller than $1$ by the choice of~$m$.  Thus,
	\begin{equation}\label{eq:smalleigen}
		\vartheta_m(G;\ww) = \lambda_1(\xX) < 1.
	\end{equation}
        Let $\zZ$, $\yY$ and $\eta$ optima for the formulation of $\vartheta_m(G; \ww)$ given in \eqref{varthetaKprimal}.  We may assume that $\yY$ is given by \eqref{eq:specdecompY}. If $\mu_1 \geq \dotsm \geq \mu_n$ are the eigenvalues of $(\zZ + \sqrt{\ww}\sqrt{\ww}^\T)$, recall from \eqref{eq:specdecompZ} and \eqref{eq:specdecompY} that $\eta = \mu_m$, and from \eqref{varthetaKoriginal} that
	\begin{equation}\label{eq:eigenineq}
		\mu_1 + \dotsb + \mu_m = \vartheta_m(G;\ww) < 1.
	\end{equation}
	Complementary slackness implies that
	\[
		\yY = \yY \xX,
	\]
        which combined with~\eqref{eq:smalleigen} yields $\yY \sqrt{\ww} = \0$. Also from complementary slackness, we have that
	\[
		(\eta \iI + \yY - (\zZ + \sqrt{\ww}\sqrt{\ww}^\T)) \xX = \0,
	\]
	and right-multiplying by $\sqrt{\ww}$, we get
	\[
		\zZ \sqrt{\ww} = (\mu_m - \ww^\T \1) \sqrt{\ww}.
	\]
	Thus the largest eigenvalue of $-\zZ$ is at least $\ww^\T \1 - \mu_m$, and $\zZ$ and $\zZ + \sqrt{\ww}\sqrt{\ww}^\T$ have a common basis of eigenvectors. Hence
	\[
		\lambda_1(-\zZ) + \sum_{i = 1}^{m-1} \lambda_{n+1-i}(-\zZ) \geq 1 - \mu_m - \sum_{i = 1}^{m-1} \mu_{i} > 0,
	\]
	where the last inequality follows from \eqref{eq:eigenineq}.

	Assume now $m = h(G) - 1 \geq 1$, and let $\zZ$ be a symmetric matrix with $\zZ_{ij} = 0$ for all $ij \notin E(G)$ such that
        \begin{equation}
          \label{eq:4}
          \lambda_1(\zZ) + \sum_{i = 1}^{m-1} \lambda_{n+1-i}(\zZ) > 0.
        \end{equation}
	Upon conjugating $\zZ$ by a suitable diagonal matrix of entries $\pm 1$, we can assume there is an eigenvector $\vv$ corresponding to $\lambda_1(\zZ)$ with $\vv \geq \0$ (note that $\zZ$ conjugated by the diagonal matrix still satisfies the hypotheses and inequality above). Let $\ww \geq \0$ with $\sqrt{\ww}$ a scalar multiple of $\vv$ scaled so that ${\ww^\T \1 - \lambda_1(\zZ) \geq -\lambda_n(\zZ)}$. Thus,
	\[
		\vartheta_m(G;\ww) \leq \sum_{i = 1}^m \lambda_i (\sqrt{\ww}\sqrt{\ww}^\T - \zZ) = \ww^\T \1 - \lambda_1(\zZ) - \sum_{i = 1}^{m-1} \lambda_{n+1-i}(\zZ) < \ww^\T \1.
	\]

	Hence, by Lemma \ref{lem:monotonicity}, the inequality~\eqref{eq:3} follows.
\end{proof}

Note that the proof of~\eqref{eq:3} shows that, if there is $\zZ \in \Sds^V$ such that $\zZ_{ij} = 0$ for each $ij \notin E$ and~\eqref{eq:4} holds, then there is $\ww \geq \0$ such that \(\vartheta_m(G;\ww) < \ww^\T \1\).

We now apply our technology to show that $h(G)$ is no better that $\lceil \chi_f(G) \rceil$ to lower bound $\chi(G)$.

\begin{theorem}
  \label{thm:2}
	For a graph $G$ and $k = \chi_f(G)$, one has $\vartheta_k(G;\ww) \geq \ww^\T \1$ for each $\ww \geq \0$.
\end{theorem}

\begin{proof}
	Let $\ww \in \Rds_+^V$. Assume $\nN$ is the vertex-coclique incidence matrix of $G$, and let $\yy$ be a nonnegative vector such that $\nN \yy = \1$ and
\[
	\chi_f(G) = \1^\T \yy.
\]
Assume the matrix $\ov{\nN}_\ww$ is obtained from $\Diag(\sqrt{\ww})\nN$ upon normalizing each column. We then define $\mM = \ov{\nN}_\ww \Diag(\yy) \ov{\nN}_\ww^\T$. If $\cal S$ is the set of all cocliques, $\xx_S$ stands for the characteristic vector of a coclique $S$, and $\ww(S) = \ww^\T \xx_S$ is the weight of a coclique $S$, then
\[
	\mM = \sum_{\substack{S \in \mathcal{S} \\[2pt] \ww(S) \neq 0}} \yy(S) \frac{\Diag(\xx_S) \sqrt{\ww}\sqrt{\ww}^\T \Diag(\xx_S)}{\ww(S)}.
\]
We now proceed to verify the following claims:
\begin{itemize}
	\item $\tr \mM = \chi_f(G)$.
	\item[] This follows from $\tr \mM = \sum_{S \in \mathcal{S}} \yy(S) = \1^\T \yy = \chi_f(G)$.
	\item $\0 \preccurlyeq \mM \preccurlyeq \iI$.
	\item[] The matrix $\mM$ is positive semidefinite and non-negative, with positive diagonal entries so as long the corresponding entry in $\ww$ is non-zero. Thus any eigenvector which is strictly positive in the support of $\ww$ must correspond to a largest eigenvalue, as consequence of the Perron-Frobenius theory. The claim now follows from
	\[
		\mM \sqrt{\ww} = \sum_{S \in \mathcal{S}} \yy(S) \cdot \Diag(\xx_S) \sqrt{\ww} = \sqrt{\ww},
	\]
	as $\nN \yy = \1$.
	\item $\langle \sqrt{\ww}\sqrt{\ww}^\T, \mM \rangle = \ww^\T \1$.
	\item[] Immediate from the claim above.
\end{itemize}
The three claims and the immediate observation that no off-diagonal entry of $\mM$ corresponding to an edge is non-zero, imply that $\mM$ is feasible for \eqref{varthetaKdual} with objective value $\geq \ww^\T \1$.
\end{proof}

\begin{theorem}
  \label{thm:3}
	For a graph $G$, one has
        \[
          \min\{ k : \forall~\ww \geq \0, \vartheta_k(G;\ww)\} \leq \lceil \chi_f(G) \rceil.
        \]
        In particular,
        \[
          h(G) \leq \lceil \chi_f(G) \rceil.
        \]
\end{theorem}

\begin{proof}
         For $m = \chi_f(G)$, one has $\vartheta_m(G;\ww) \geq \ww^\T \1$ for every $\ww \geq \0$ by Theorem~\ref{thm:2}.  For $k = \lceil \chi_f(G) \rceil$, by Lemma~\ref{lem:monotonicity} one has $\vartheta_k(G;\ww) \geq \ww^\T \1$ for every $\ww \geq \0$.  This proves the first inequality in the statement.  The second inequality follows from the first, together with the second inequality from Theorem~\ref{thm:maxtheta}.
\end{proof}

As a consequence of the preceding results and proofs, we obtain the following corollary.

\begin{corollary}

For every graph $G$ and for any $\xx \geq \0$, the smallest integer $k \geq 1$ so that $\vartheta_k(G;\xx) \geq \xx^\T \1$ is such that
\[
	k \leq \lceil \chi_f(G) \rceil.
\]
Moreover, there is a vector $\ww \geq \0$ such that the smallest integer $k \geq 1$ which gives $\vartheta_k(G;\ww) \geq \ww^\T \1$ is such that
\[
	\lceil \vartheta(\ov{G}) \rceil \leq k.
\]
\end{corollary}

\begin{proof}
  The smallest integer $k$ such that $\vartheta_k(G;\xx) \geq \xx^\T \1$ is smaller than or equal to the $\min$ in Theorem~\ref{thm:3}. Assume that $G$ has at least one edge. By the remark following the proof Theorem~\ref{thm:1}, there is $\zZ \in \Sds^V$ such that $\zZ_{ij} = 0$ for each $ij \notin E$ and~\eqref{eq:1} holds, with $m = \lceil \vartheta(\ov{G}) \rceil -1$.  The remark following the proof of Theorem~\ref{thm:maxtheta} shows that there is \(\ww \geq \0\) such that \(\vartheta_m(G;\ww) < \ww^\T \1\).  This proves the second statement.
\end{proof}

In particular, the case $\xx = \1$ giving a lower bound for $\chi(G)$ is a result obtained by Mohar and Poljak \cite[Corollary 4.16]{MoharPoljak}.

\section{Future work}

Mančinska and Roberson \cite{MancinskaRoberson} defined the projective rank $\xi_f(G)$ and showed that it lower bounds both $\chi_q(G)$ and $\chi_f(G)$. As seen in the work of Wocjan and Elphick \cite{WocjanElphick}, several spectral techniques that are typically applied to lower bound $\chi(G)$ also apply for $\xi_f(G)$, or its related parameter $\xi(G)$, the orthogonal rank. We therefore wonder if $h(G)$ lower bounds $\xi_f(G)$, or at least if it lower bounds $\xi(G)$ or $\chi_q(G)$, the quantum chromatic number.

Sum-of-eigenvalues bounds on graph parameters related to partitions are a common theme in combinatorics, see for instance the work of Bollobás and Nikiforov \cite{BollobasNikiforovInterlacing} and also Wocjan and Elphick \cite{WocjanElphick2}. We believe there is a common framework under which several of these bounds should fall, and that their optimized versions are usually possible upon varying the non-zero or diagonal entries of the adjacency matrix. The connection with semidefinite programming generally appears at this point. We see our work as a modest attempt to unveil some of the deeper connections we envision --- for example, our work implies that a ``fractional version" of interlacing is not only possible but also useful, and we wonder what other application it might encounter given how ubiquitous interlacing is in spectral graph theory.

Two of the authors of this work have recently coauthored the paper \cite{Benedetto} which formalizes a duality connection between Hoffman's bound $\chi(G) \geq 1 - (\lambda_1/\lambda_n)$ and an eigenvalue bound for $\alpha(G)$ also due to Hoffman and, independently, to Delsarte. See \cite{GodsilNewman} for an extensive treatment of this bound and some of its generalizations. We believe that the same duality should exist for $\vartheta_k(G;\ww)$ and if so, we ask if it leads to some new and interesting bounds to $\alpha_k(G)$. 

\section*{Acknowledgments}

This work was partially supported by Conselho Nacional de Desenvolvimento Científico e Tecnológico – CNPq (Proc.423833/2018-9). GC would like to acknowledge Monique Laurent for fruitful conversations about the topic of this paper during the workshop ``Analytical and combinatorial aspects of quantum information theory'' at the International Centre for Mathematical Sciences - Edinburgh, in September 2019. GC also acknowledges the support offered by ICMS to attend the aforementioned workshop. RG acknowledges the Master's scholarship granted by CNPq from 2019 to 2021. Finally, all authors acknowledge the extremely helpful comments from Clive Elphick upon the first version of this paper.


\end{document}